\documentclass[fleqn]{amsart}

\usepackage{amssymb,bm,mathrsfs}
\usepackage{hyperref}

\newtheoremstyle{conv}{}{}{\upshape}{}{\itshape}{}{ }{}
\newtheoremstyle{note}{}{}{\itshape}{}{\itshape}{}{ }{}

\makeatletter
\def\thmhead#1#2#3{%
	\thmnumber{\textup{\mdseries(#2)}}%
	\thmname{\@ifnotempty{#2}{~}#1}%
	\thmnote{{\the\thm@notefont(#3)}}}
\makeatother

\newtheorem{thm}[equation]{Theorem}
\newtheorem{prop}[equation]{Proposition}
\theoremstyle{note}
\newtheorem{note}[equation]{}
\theoremstyle{conv}
\newtheorem{conv}[equation]{}
\theoremstyle{definition}

\theoremstyle{remark}
\newtheorem{rem}[equation]{Remark}
\newtheorem*{ack}{Acknowledgments}

\newenvironment{axioms}[1]
{\begin{list}{(#1\arabic{enumi})}{
\usecounter{enumi}%
\def\makelabel##1{\hspace\labelsep \upshape##1}%
\setlength{\labelwidth}{\leftmargin}%
\advance \labelwidth-\labelsep}}
{\end{list}}

\renewenvironment{description}
{\begin{list}{}{%
\renewcommand{\makelabel}[1]{\hspace\labelsep \upshape ##1}%
\setlength{\labelwidth}{\leftmargin}%
\advance \labelwidth-\labelsep}}
{\end{list}}

\numberwithin{equation}{section}

\newcommand{\hypergeometricseries}[5]{{}_{#1}F_{#2}\!\left(\left.\!\!\!\begin{array}{c} #3 \\ #4 \end{array}\!\!\right| #5 \right)}

\allowdisplaybreaks[4]

\begin{document}

\title{A bilinear form relating two Leonard systems}
\author{Hajime Tanaka}
\address{Graduate School of Information Sciences, Tohoku University, Sendai 980-8579, Japan}
\email{htanaka@math.is.tohoku.ac.jp}
\dedicatory{Dedicated to Professor Tatsuro Ito on the occasion of his $60$th birthday}
\date{}
\begin{abstract}
Let $\Phi$, $\Phi'$ be Leonard systems over a field $\mathbb{K}$, and $V$, $V'$ the vector spaces underlying $\Phi$, $\Phi'$, respectively.
In this paper, we introduce and discuss a \emph{balanced bilinear form} on $V\times V'$.
Such a form naturally arises in the study of $Q$-polynomial distance-regular graphs.
We characterize a balanced bilinear form from several points of view.
\end{abstract}

\maketitle

\section{Introduction}\label{sec: introduction}

Leonard systems naturally arise in representation theory, combinatorics, and the theory of orthogonal polynomials (see e.g. \cite{Terwilliger2003JCAM,Terwilliger2006N}).
Hence they are receiving considerable attention.
Indeed, the use of the name `Leonard system' is motivated by a connection to a theorem of Leonard \cite{Leonard1982SIAM}, \cite[p.~260]{BI1984B}, which involves the $q$-Racah polynomials \cite{AW1979SIAM} and some related polynomials of the Askey scheme \cite{KS1998R}.

Let $\Phi$, $\Phi'$ be Leonard systems over a field $\mathbb{K}$, and $V$, $V'$ the vector spaces underlying $\Phi$, $\Phi'$, respectively (\S \ref{sec: Leonard systems}).
Suppose $\dim V'\leqslant\dim V$.
We consider a situation where $\Phi$, $\Phi'$ are related by means of a bilinear form $(\cdot|\cdot):V\times V'\rightarrow\mathbb{K}$ satisfying certain orthogonality conditions (see \S \ref{sec: balanced bilinear form} for the precise definition).
In this case we say that $(\cdot|\cdot)$ is \emph{balanced} with respect to $\Phi$, $\Phi'$, and call $\Phi'$ a \emph{descendent} of $\Phi$.
The notion of a balanced bilinear form originates in the theory of $Q$-\emph{polynomial distance-regular graphs} \cite{BI1984B,BCN1989B,Godsil1993B}.
Specifically, such a form arises in the context of subsets having minimal \emph{width} and \emph{dual width} \cite{BGKM2003JCTA,Tanaka2006JCTA,HS2007EJC} and in the context of certain irreducible modules for the \emph{Terwilliger algebra} \cite{Terwilliger1992JAC,Terwilliger1993JACa,Terwilliger1993JACb}.
For example, let $V$ be the primary module of the hypercube $Q_d$ with respect to a base vertex $x$ (where $\mathbb{K}=\mathbb{R}$, say).
For the former context, let $V'$ be the primary module of an induced subgraph $Q_{d'}$ $(d'\leqslant d)$ containing $x$.
For the latter, let $V'$ be an irreducible module with respect to another base vertex $y$, and suppose that $V'$ has endpoint $\partial(x,y)$ and is not orthogonal to $V$.
In each case, let $\Phi$, $\Phi'$ be the Leonard systems associated with $V$, $V'$ (cf. \cite{Terwilliger1993JACa}, \cite[Example 1.4]{ITT2001P}), and for the latter we further replace $\Phi$, $\Phi'$ by their `duals'.
Then the restriction of the standard inner product onto $V\times V'$ turns out to be balanced with respect to $\Phi$, $\Phi'$.
See \cite{Tanaka2008pre} for details.
We believe that the study of a balanced bilinear form will lead to a unification of these two approaches at a certain level and thus help better understand the structure of $Q$-polynomial distance-regular graphs.

The contents of the paper are as follows.
\S \ref{sec: Leonard systems} reviews basic terminology, notation  and facts concerning Leonard systems.
In \S \ref{sec: balanced bilinear form} we introduce a balanced bilinear form as well as a descendent.
\S\S \ref{sec: properties of bilinear form} and \ref{sec: reconstruction of bilinear form} are devoted to its properties and a characterization in terms of the \emph{parameter arrays} of $\Phi$, $\Phi'$ (Theorem \eqref{characterization of bilinear form in terms of parameter arrays}).
It should be remarked that the isomorphism class of a Leonard system is determined by its parameter array (\cite[Theorem 1.9]{Terwilliger2001LAA}).
\S \ref{sec: bilinear form in parametric form} establishes a classification of the descendents of Leonard systems (Theorem \eqref{characterization of bilinear form in parametric form}).
\S \ref{sec: Phi' in terms of bilinear form} deals with a `converse' problem:
given the Leonard system $\Phi$ and a bilinear form $(\cdot|\cdot): V\times V'\rightarrow\mathbb{K}$, we ask whether there is a descendent $\Phi'$ defined on $V'$ so that $(\cdot|\cdot)$ is balanced with respect to $\Phi$, $\Phi'$.
Theorem \eqref{characterization of Phi' in terms of bilinear form} is the main result on this topic.
\S \ref{sec: polynomials} discusses an interpretation of a balanced bilinear form as an orthogonality of some polynomials of the Askey scheme.
The paper ends with an appendix containing a list of the parameter arrays of the Leonard systems.
We shall apply these results to the study of $Q$-polynomial distance-regular graphs in future papers.

\section{Leonard systems}\label{sec: Leonard systems}

Let $\mathbb{K}$ be a field, $d$ a \emph{positive} integer, $\mathscr{A}$ a $\mathbb{K}$-algebra isomorphic to the full matrix algebra $\mathrm{Mat}_{d+1}(\mathbb{K})$, and $V$ an irreducible left $\mathscr{A}$-module.
We remark that $V$ is unique up to isomorphism, and that $V$ has dimension $d+1$.
An element $A$ of $\mathscr{A}$ is said to be \emph{multiplicity-free} if it has $d+1$ mutually distinct eigenvalues in $\mathbb{K}$.
Let $A$ be a multiplicity-free element of $\mathscr{A}$ and $\{\theta_i\}_{i=0}^d$ an ordering of the eigenvalues of $A$.
Then by elementary linear algebra there is a sequence of elements $\{E_i\}_{i=0}^d$ in $\mathscr{A}$ such that (i) $AE_i=\theta_iE_i$; (ii) $E_iE_j=\delta_{ij}E_i$; (iii) $\sum_{i=0}^dE_i=I$ where $I$ is the identity of $\mathscr{A}$.
We call $E_i$ the \emph{primitive idempotent} of $A$ associated with $\theta_i$.

A \emph{Leonard system} in $\mathscr{A}$ (\cite[Definition 1.4]{Terwilliger2001LAA}) is a sequence
\begin{equation}\label{Leonard system}
	\Phi=\left(A;A^*;\{E_i\}_{i=0}^d;\{E_i^*\}_{i=0}^d\right)
\end{equation}
satisfying the following axioms (LS1)--(LS5):
\begin{axioms}{LS}
\item Each of $A,A^*$ is a multiplicity-free element in $\mathscr{A}$.
\item $\{E_i\}_{i=0}^d$ is an ordering of the primitive idempotents of $A$.
\item $\{E_i^*\}_{i=0}^d$ is an ordering of the primitive idempotents of $A^*$.
\item $E_i^*AE_j^*=\begin{cases} 0 & \text{if } |i-j|>1 \\ \ne 0 & \text{if } |i-j|=1 \end{cases} \quad (0\leqslant i,j\leqslant d)$.
\item $E_iA^*E_j=\begin{cases} 0 & \text{if } |i-j|>1 \\ \ne 0 & \text{if } |i-j|=1 \end{cases} \quad (0\leqslant i,j\leqslant d)$.
\end{axioms}
We call $d$ the \emph{diameter} of $\Phi$, and say that $\Phi$ is \emph{over} $\mathbb{K}$.
For notational convenience, we define $E_i=E_i^*=0$ if $i<0$ or $i>d$.
We refer the reader to \cite{Terwilliger1992JAC,Terwilliger2001LAA,Terwilliger2004LAA,Terwilliger2005DCC,Terwilliger2006N} for background on Leonard systems.

A Leonard system $\Psi$ in a $\mathbb{K}$-algebra $\mathscr{B}$ is \emph{isomorphic} to $\Phi$ if there is a $\mathbb{K}$-algebra isomorphism $\gamma:\mathscr{A}\rightarrow\mathscr{B}$ such that $\Psi=\Phi^{\gamma}:=\left(A^{\gamma};A^{*\gamma};\{E_i^{\gamma}\}_{i=0}^d;\{E_i^{*\gamma}\}_{i=0}^d\right)$.
Let $\xi,\xi^*,\zeta,\zeta^*$ be scalars in $\mathbb{K}$ such that $\xi\ne 0$, $\xi^*\ne 0$.
Then
\begin{equation}
	\left(\xi A+\zeta I;\xi^*A^*+\zeta^*I;\{E_i\}_{i=0}^d;\{E_i^*\}_{i=0}^d\right)
\end{equation}
is a Leonard system in $\mathscr{A}$, called an \emph{affine transformation} of $\Phi$.
We say that $\Phi$, $\Psi$ are \emph{affine-isomorphic} if $\Psi$ is isomorphic to an affine transformation of $\Phi$.
Also
\begin{align}
	\Phi^*&=\left(A^*;A;\{E_i^*\}_{i=0}^d;\{E_i\}_{i=0}^d\right), \\
	\Phi^{\downarrow}&=\left(A;A^*;\{E_i\}_{i=0}^d;\{E_{d-i}^*\}_{i=0}^d\right), \\
	\Phi^{\Downarrow}&=\left(A;A^*;\{E_{d-i}\}_{i=0}^d;\{E_i^*\}_{i=0}^d\right)
\end{align}
are Leonard systems in $\mathscr{A}$.
Viewing $*,\downarrow,\Downarrow$ as permutations on all Leonard systems,
\begin{equation*}
	*^2=\downarrow^2=\Downarrow^2=1, \quad \Downarrow *=*\downarrow, \quad \downarrow *=*\Downarrow, \quad \downarrow\Downarrow=\Downarrow\downarrow.
\end{equation*}
The group generated by the symbols $*,\downarrow,\Downarrow$ subject to the above relations is the dihedral group $D_4$ with $8$ elements.
For the rest of this paper we shall use the following notational convention:

\begin{conv}
For any $g\in D_4$ and for any object $f$ associated with $\Phi$, we let $f^g$ denote the corresponding object for $\Phi^{g^{-1}}$; an example is $E_i^*(\Phi)=E_i(\Phi^*)$.
\end{conv}

For $0\leqslant i\leqslant d$ let $\theta_i$ (resp. $\theta_i^*$) be the eigenvalue of $A$ (resp. $A^*$) associated with $E_i$ (resp. $E_i^*$).
By \cite[Theorem 3.2]{Terwilliger2001LAA} there are scalars $\varphi_i$ $(1\leqslant i\leqslant d)$ in $\mathbb{K}$ and a $\mathbb{K}$-algebra isomorphism $\natural:\mathscr{A}\rightarrow\mathrm{Mat}_{d+1}(\mathbb{K})$ such that
\begin{equation*}
	A^{\natural}=\begin{pmatrix} \theta_0 &&&&& \bm{0} \\ 1 & \theta_1 \\ & 1 & \theta_2 \\ && \cdot & \cdot \\ &&& \cdot & \cdot \\ \bm{0} &&&& 1 & \theta_d \end{pmatrix}, \quad A^{*\natural}=\begin{pmatrix} \theta_0^* & \varphi_1 &&&& \bm{0} \\ & \theta_1^* & \varphi_2 \\ && \theta_2^* & \cdot \\ &&& \cdot & \cdot \\ &&&& \cdot & \varphi_d \\ \bm{0} &&&&& \theta_d^* \end{pmatrix}.
\end{equation*}
We define $\phi_i=\varphi_i^{\Downarrow}$ $(1\leqslant i\leqslant d)$.
The \emph{parameter array} of $\Phi$ is the sequence
\begin{equation}
	p(\Phi)=\left(\{\theta_i\}_{i=0}^d;\{\theta_i^*\}_{i=0}^d;\{\varphi_i\}_{i=1}^d;\{\phi_i\}_{i=1}^d\right).
\end{equation}
Terwilliger \cite[Theorem 1.9]{Terwilliger2001LAA} showed that the isomorphism class of $\Phi$ is determined by $p(\Phi)$, and that the set of parameter arrays of Leonard systems over $\mathbb{K}$ with diameter $d$ is characterized by the following properties \eqref{PA1}--\eqref{PA5}:
\begin{gather}
	\varphi_i\ne 0, \quad \phi_i\ne 0 \quad (1\leqslant i\leqslant d). \label{PA1} \\
	\theta_i\ne\theta_j, \quad \theta_i^*\ne\theta_j^* \quad \text{\emph{if}} \ i\ne j \quad (0\leqslant i,j\leqslant d). \label{PA2} \\
	\varphi_i=\phi_1\sum_{\ell=0}^{i-1}\frac{\theta_{\ell}-\theta_{d-\ell}}{\theta_0-\theta_d}+(\theta_i^*-\theta_0^*)(\theta_{i-1}-\theta_d) \quad (1\leqslant i\leqslant d). \label{PA3} \\
	\displaystyle\phi_i=\varphi_1\sum_{\ell=0}^{i-1}\frac{\theta_{\ell}-\theta_{d-\ell}}{\theta_0-\theta_d}+(\theta_i^*-\theta_0^*)(\theta_{d-i+1}-\theta_0) \quad (1\leqslant i\leqslant d). \label{PA4} \\
	\frac{\theta_{i-2}-\theta_{i+1}}{\theta_{i-1}-\theta_i}, \frac{\theta_{i-2}^*-\theta_{i+1}^*}{\theta_{i-1}^*-\theta_i^*} \ \text{\emph{are equal and independent of}} \ i \ (2\leqslant i\leqslant d-1). \label{PA5}
\end{gather}

It is known \cite[Theorem 6.1]{Terwilliger2004LAA} that there is a unique antiautomorphism $\dag$ of $\mathscr{A}$ such that $A^{\dag}=A$ and $A^{*\dag}=A^*$.
For the rest of this paper let $\langle\cdot,\cdot\rangle:V\times V\rightarrow\mathbb{K}$ denote a nondegenerate bilinear form on $V$ such that (\cite[\S 15]{Terwilliger2004LAA})
\begin{equation}
	\langle Xu,v\rangle=\langle u,X^{\dag}v\rangle \quad (u,v\in V,\ X\in\mathscr{A}).
\end{equation}
We shall write
\begin{equation}
	||u||^2=\langle u,u\rangle \quad (u\in V).
\end{equation}
For convenience, let $\mathscr{D}$ be the subalgebra of $\mathscr{A}$ generated by $A$.
Observe that
\begin{equation}
	\mathscr{D}=\mathrm{span}\{I,A,A^2,\dots,A^d\}=\mathrm{span}\{E_0,E_1,\dots,E_d\}.
\end{equation}
Obviously $\dagger$ fixes each element of $\mathscr{D}\cup\mathscr{D}^*$, so that we have
\begin{equation}
	\langle Xu,v\rangle=\langle u,Xv\rangle \quad (u,v\in V,\ X\in\mathscr{D}\cup\mathscr{D}^*),
\end{equation}
from which it follows that
\begin{equation}
	\langle E_iV,E_jV\rangle=\langle E_i^*V,E_j^*V\rangle=0 \quad \text{\emph{if}}\ i\ne j \quad (0\leqslant i,j\leqslant d).
\end{equation}
From now on let $u$ be a nonzero vector in $E_0V$.
Then $E_i^*u\ne 0$ for $0\leqslant i\leqslant d$ (\cite[Lemma 10.2]{Terwilliger2004LAA}), so that $\{E_i^*u\}_{i=0}^d$ is a basis of $V$.
\noindent
Let
\begin{equation}
	\nu=\mathrm{trace}(E_0^*E_0)^{-1}, \quad k_i=\mathrm{trace}(E_i^*E_0)\nu \quad (0\leqslant i\leqslant d).
\end{equation}
Indeed, by \cite[Lemma 9.2]{Terwilliger2004LAA} $\mathrm{trace}(E_i^*E_0)\ne 0$ $(0\leqslant i\leqslant d)$.
With this notation we have (\cite[Theorem 15.3]{Terwilliger2004LAA})
\begin{equation}\label{<Ei*u,Ej*u>}
	\langle E_i^*u,E_j^*u\rangle=\delta_{ij}k_i\nu^{-1}||u||^2 \quad (0\leqslant i,j\leqslant d).
\end{equation}

The \emph{dual switching element} for $\Phi$ (\cite[Note 5.2]{NT2008LAAa}) is
\begin{equation}
	S^*=\sum_{\ell=0}^d\frac{\phi_1\phi_2\dots\phi_{\ell}}{\varphi_1\varphi_2\dots\varphi_{\ell}}E_{\ell}^*.
\end{equation}
It follows that
\begin{equation}\label{switching}
	S^*E_0V=E_dV,
\end{equation}
and moreover that any element $X\in\mathscr{D}^*$ satisfying $XE_0V\subseteq E_dV$ is a scalar multiple of $S^*$ (\cite[Theorem 6.7]{NT2008LAAa}).

Finally, we remark that
\begin{equation}\label{ith component of 0*}
	E_0^*V+E_1^*V+\dots+E_i^*V=E_0^*V+AE_0^*V+\dots+A^iE_0^*V \quad (0\leqslant i\leqslant d). 
\end{equation}
In particular:
\begin{equation}\label{V=DE0*V}
	V=\mathscr{D}E_0^*V.
\end{equation}

\section{A balanced bilinear form}\label{sec: balanced bilinear form}

For the rest of this paper, we shall retain the notation of the previous section.
Except in \S \ref{sec: Phi' in terms of bilinear form} we shall always refer to the following set-up:

\begin{conv}\label{Phi, Phi'}
Let $\Phi'=\bigl(A';A^{*\prime};\{E_i'\}_{i=0}^{d'};\{E_i^{*\prime}\}_{i=0}^{d'}\bigr)$ be a Leonard system over $\mathbb{K}$ with diameter $d'\leqslant d$.
For any object $f$ associated with $\Phi$, we let $f'$ denote the corresponding object for  $\Phi'$; an example is $V'=V(\Phi')$.
\end{conv}

A nonzero bilinear form $(\cdot|\cdot): V\times V'\rightarrow\mathbb{K}$ is \emph{balanced} with respect to $\Phi$, $\Phi'$ if (B1), (B2) hold below:
\begin{axioms}{B}
\item There is an integer $\rho$ ($0\leqslant\rho\leqslant d-d'$) such that $(E_i^*V|E_j^{*\prime}V')=0$ if $i-\rho\ne j$ \ $(0\leqslant i\leqslant d, \ 0\leqslant j\leqslant d')$.
\item $(E_iV|E_j'V')=0$ if $i<j$ or $i>j+d-d'$ \ $(0\leqslant i\leqslant d, \ 0\leqslant j\leqslant d')$.
\end{axioms}
We call $\rho$ the \emph{endpoint} of $(\cdot|\cdot)$ (with respect to $\Phi$, $\Phi'$), and refer to the form also as $\rho$-\emph{balanced}.
We say that $\Phi'$ is a $\rho$-\emph{descendent} (or simply a \emph{descendent}) of $\Phi$ whenever such a form exists.

\begin{rem}\label{balanced with respect to Leonard pairs}
A bilinear form $(\cdot|\cdot): V\times V'\rightarrow\mathbb{K}$ which is balanced with respect to $\Phi$, $\Phi'$ is in fact balanced with respect to the following pairs of Leonard systems:
\begin{center}
\begin{tabular}{c|c}
pair & endpoint \\
\hline
$\Phi,\Phi'$ & $\rho$ \\
$\Phi^{\downarrow},\Phi^{\prime\downarrow}$ & $d-d'-\rho$ \\
$\Phi^{\Downarrow},\Phi^{\prime\Downarrow}$ & $\rho$ \\
$\Phi^{\downarrow\Downarrow},\Phi^{\prime\downarrow\Downarrow}$ & $d-d'-\rho$
\end{tabular}
\end{center}
\end{rem}

\begin{rem}\label{transitivity}
If $\Phi'$ is a $\rho$-descendent of $\Phi$, then any $\rho'$-descendent $\Phi''$ of $\Phi'$ is a $(\rho+\rho')$-descendent of $\Phi$.
Indeed, let $(\cdot|\cdot): V\times V'\rightarrow\mathbb{K}$ and $(\cdot|\cdot)': V'\times V''\rightarrow\mathbb{K}$ be corresponding balanced bilinear forms, where $V''=V(\Phi'')$.
Let $\mathrm{proj}'':V''\rightarrow V'$ be a unique linear map such that $(v'|v'')'=\langle v',\mathrm{proj}''v''\rangle'$ for all $v'\in V', v''\in V''$.
Then the bilinear form $V\times V''\rightarrow\mathbb{K}$ defined by $(v,v'')\mapsto (v|\mathrm{proj}''v'')$ $(v\in V, v''\in V'')$ is $(\rho+\rho')$-balanced with respect to $\Phi$, $\Phi''$.
(That this form is nonzero follows e.g. from \eqref{values of products} below.)
\end{rem}

\begin{rem}
If $\Phi'$ is a descendent of $\Phi$, then any Leonard system affine isomorphic to $\Phi'$ is a descendent of any Leonard system affine isomorphic to $\Phi$.
Later we shall show that a balanced bilinear form has full-rank (cf. \eqref{values of products}), from which it follows that two Leonard systems are descendents of each other if and only if they are affine isomorphic (cf. \eqref{characterization of Phi' in terms of bilinear form}).
Now, let $[\Phi]$, $[\Phi']$ denote the affine-isomorphism classes of $\Phi$, $\Phi'$, respectively, and write $[\Phi']\preccurlyeq [\Phi]$ if $\Phi'$ is a descendent of $\Phi$.
Then, in view of \eqref{transitivity} and the above comments, $\preccurlyeq$ defines a well-defined poset structure on the set of affine-isomorphism classes of Leonard systems over $\mathbb{K}$.
\end{rem}

\section{Properties of a balanced bilinear form}\label{sec: properties of bilinear form}

In this section,  we shall study the basic properties of a balanced bilinear form.
With reference to \eqref{Phi, Phi'}, we shall assume that there is a bilinear form $(\cdot|\cdot): V\times V'\rightarrow\mathbb{K}$ which is $\rho$-balanced with respect to $\Phi$, $\Phi'$ ($0\leqslant\rho\leqslant d-d'$).

We define a $\mathbb{K}$-algebra homomorphism $\sigma:\mathscr{D}^*\rightarrow\mathscr{D}^{*\prime}$ by
\begin{equation}
	E_i^{*\sigma}=E_{i-\rho}^{*\prime} \quad (0\leqslant i\leqslant d).
\end{equation}
Clearly $\sigma$ is surjective.
By (B1) it follows that
\begin{equation}\label{sigma commutes with product}
	(Xv|v')=(v|X^{\sigma}v') \quad (v\in V,\ v'\in V',\ X\in\mathscr{D}^*).
\end{equation}
Let $\mathrm{proj}:V\rightarrow V'$ and $\mathrm{proj}':V'\rightarrow V$ be unique linear maps satisfying
\begin{equation}
	(v|v')=\langle \mathrm{proj}\,v,v'\rangle'=\langle v,\mathrm{proj}'v'\rangle \quad (v\in V,\ v'\in V').
\end{equation}
It follows from (B2) that $\mathrm{proj}\,E_0V\subseteq E_0'V'$.
Furthermore, by \eqref{V=DE0*V} we have
\begin{equation*}
	(E_0V|V')=(E_0V|\mathscr{D}^{*\sigma}V')=(\mathscr{D}^*E_0V|V')=(V|V')\ne 0,
\end{equation*}
so that $\mathrm{proj}\,E_0V\ne 0$.
Hence
\begin{equation}\label{projE0V=E0'V'}
	\mathrm{proj}\,E_0V=E_0'V'.
\end{equation}
With this explained, we have

\begin{prop}\label{values of products}
Let $u$ (resp. $u'$) be a nonzero vector in $E_0V$ (resp. $E_0'V'$).
Then there is a nonzero scalar $\epsilon\in\mathbb{K}$ such that
\begin{equation*}
	(E_i^*u|E_j^{*\prime}u')=\epsilon\delta_{i-\rho,j}k_j'\nu^{\prime-1}||u'||^{\prime 2} \quad (0\leqslant i\leqslant d, \ 0\leqslant j\leqslant d').
\end{equation*}
In particular, $(\cdot|\cdot)$ has full-rank, i.e., $\mathrm{proj}': V'\rightarrow V$ is injective.
\end{prop}

\begin{proof}
By \eqref{projE0V=E0'V'} there is a nonzero scalar $\epsilon\in\mathbb{K}$ such that $\mathrm{proj}\,u=\epsilon u'$.
Hence from \eqref{<Ei*u,Ej*u>} it follows that
\begin{equation*}
	(E_i^*u|E_j^{*\prime}u')=\epsilon\delta_{i-\rho,j}\langle u',E_j^{*\prime}u'\rangle'=\epsilon\delta_{i-\rho,j}k_j'\nu^{\prime-1}||u'||^{\prime 2}. \qedhere
\end{equation*}
\end{proof}

\begin{prop}\label{A*', S*' in terms of A*sigma, S*sigma}
The following \textup{(i)}, \textup{(ii)} hold:
\begin{enumerate}
\item There are scalars $\xi^*,\zeta^*\in\mathbb{K}$ such that $A^{*\prime}=\xi^*A^{*\sigma}+\zeta^*I'$, where $I'$ is the identity of $\mathscr{A}'$.
\item $S^{*\prime}=\dfrac{\varphi_1\varphi_2\dots\varphi_{\rho}}{\phi_1\phi_2\dots\phi_{\rho}}S^{*\sigma}$.
\end{enumerate}
\end{prop}

\begin{proof}
(i):
For $2\leqslant i\leqslant d'$ we have
\begin{align*}
	\langle A^{*\sigma}E_0'V',E_i'V'\rangle' &= \langle E_0'V',A^{*\sigma}E_i'V'\rangle' \\
		&= (E_0V|A^{*\sigma}E_i'V') && \text{by}\ \eqref{projE0V=E0'V'} \\
		&= (A^*E_0V|E_i'V') \\
		&= 0
\end{align*}
since $A^*E_0V\subseteq E_0V+E_1V$; it follows that $A^{*\sigma}E_0'V'\subseteq E_0'V'+E_1'V'$.
Now it is clear from \eqref{ith component of 0*} that $A^{*\sigma}\in\mathrm{span}\{A^{*\prime},I'\}$.

(ii):
For $0\leqslant i\leqslant d'-1$ we have
\begin{align*}
	\langle S^{*\sigma}E_0'V',E_i'V'\rangle' &= \langle E_0'V',S^{*\sigma}E_i'V'\rangle' \\
	&= (E_0V|S^{*\sigma}E_i'V') && \text{by}\ \eqref{projE0V=E0'V'} \\
	&= (S^*E_0V|E_i'V') \\
	&= (E_dV|E_i'V') && \text{by}\ \eqref{switching} \\
	&= 0,
\end{align*}
so that $S^{*\sigma}E_0'V'\subseteq E_{d'}'V'$.
Hence (\cite[Theorem 6.7]{NT2008LAAa})  $S^{*\sigma}$ is a scalar multiple of $S^{*\prime}$, and the scalar factor is given by comparing the coefficient of $E_0^{*\prime}$ in $S^{*\prime}$, $S^{*\sigma}$.
\end{proof}


\section{Reconstruction of the balanced bilinear form}\label{sec: reconstruction of bilinear form}

In this section, we shall see that \eqref{A*', S*' in terms of A*sigma, S*sigma} turns out to give a necessary and sufficient condition on the existence of a balanced bilinear form.
With reference to \eqref{Phi, Phi'}, let $\rho$ be an integer such that $0\leqslant\rho\leqslant d-d'$.
As in \S \ref{sec: properties of bilinear form}, define a $\mathbb{K}$-algebra homomorphism $\sigma:\mathscr{D}^*\rightarrow\mathscr{D}^{*\prime}$ by
\begin{equation}
	E_i^{*\sigma}=E_{i-\rho}^{*\prime} \quad (0\leqslant i\leqslant d).
\end{equation}
We shall assume (i), (ii) below:
\begin{enumerate}
\item There are scalars $\xi^*,\zeta^*\in\mathbb{K}$ such that $A^{*\prime}=\xi^*A^{*\sigma}+\zeta^*I'$.
\item $S^{*\prime}=\dfrac{\varphi_1\varphi_2\dots\varphi_{\rho}}{\phi_1\phi_2\dots\phi_{\rho}}S^{*\sigma}$.
\end{enumerate}

Let $u$ (resp. $u'$) be a nonzero vector in $E_0V$ (resp. $E_0'V'$).
We define a bilinear form $(\cdot|\cdot): V\times V'\rightarrow\mathbb{K}$ by
\begin{equation}\label{reconstruction}
	(E_i^*u|E_j^{*\prime}u')=\delta_{i-\rho,j}k_j'\nu^{\prime-1}||u'||^{\prime 2} \quad (0\leqslant i\leqslant d,\ 0\leqslant j\leqslant d').
\end{equation}
Clearly we have
\begin{equation}
	(Xv|v')=(v|X^{\sigma}v') \quad (v\in V,\ v'\in V',\ X\in\mathscr{D}^*).
\end{equation}
It follows that
\begin{note}\label{bilinear form is balanced}
The following \textup{(i)}, \textup{(ii)} hold:
\begin{enumerate}
\item $(E_0V+\dots+E_{i-1}V|E_i'V'+\dots+E_{d'}'V')=0$ \ $(1\leqslant i\leqslant d')$.
\item $(E_{i+d-d'}V+\dots+E_dV|E_0'V'+\dots+E_{i-1}'V')=0$ \ $(1\leqslant i\leqslant d')$.
\end{enumerate}
In other words, $(\cdot|\cdot)$ satisfies \textup{(B2)} and is therefore balanced with respect to $\Phi$, $\Phi'$.
\end{note}

\begin{proof}
(i):
Suppose to begin with that $i=1$.
From \eqref{reconstruction} and \eqref{<Ei*u,Ej*u>} we have
\begin{equation*}
	(u|E_j^{*\prime}u')=\langle u',E_j^{*\prime}u'\rangle' \quad (0\leqslant j\leqslant d'),
\end{equation*}
from which it follows that $(u|v')=\langle u',v'\rangle'$ for \emph{all} $v'\in V'$.
Thus we obtain
\begin{equation*}
	(E_0V|E_1'V'+\dots+E_{d'}'V')=\langle E_0'V',E_1'V'+\dots+E_{d'}'V'\rangle'=0.
\end{equation*}
For $1<i\leqslant d'$ we now compute
\begin{align*}
	(E_0V+\dots+E_{i-1}V|E_i'V'+\dots+E_{d'}'V') \hspace{-3cm} \\
	&=\sum_{k=0}^{i-1}\sum_{\ell=0}^{d'-i}(A^{*k}E_0V|A^{*\sigma\ell}E_{d'}'V') \\
	&=\sum_{k=0}^{i-1}\sum_{\ell=0}^{d'-i}(E_0V|(A^{*\sigma})^{k+\ell}E_{d'}'V') \\
	&=(E_0V|E_1'V'+\dots+E_{d'}'V') \\
	&=0.
\end{align*}

(ii):
From \eqref{switching} it follows that
\begin{align*}
	E_{i+d-d'}V+\dots+E_dV&=E_dV+A^*E_dV+\dots+(A^*)^{d'-i}E_dV \\
	&=S^*(E_0V+A^*E_0V+\dots+(A^*)^{d'-i}E_0V) \\
	&=S^*(E_0V+\dots+E_{d'-i}V).
\end{align*}
Likewise we have
\begin{align*}
	E_0'V'+\dots+E_{i-1}'V'&=(S^{*\prime})^{-1}(E_{d'-i+1}'V'+\dots+E_{d'}'V') \\
	&=(S^{*\sigma})^{-1}(E_{d'-i+1}'V'+\dots+E_{d'}'V').
\end{align*}
Hence the result follows from (i).
\end{proof}

Finally, we may conveniently summarize \eqref{A*', S*' in terms of A*sigma, S*sigma} and \eqref{bilinear form is balanced} in the following form:

\begin{thm}\label{characterization of bilinear form in terms of parameter arrays}
With reference to \eqref{Phi, Phi'}, $\Phi'$ is a $\rho$-descendent of $\Phi'$ if and only if the parameter arrays of $\Phi$, $\Phi'$ satisfy \textup{(i)}, \textup{(ii)} below:
\begin{enumerate}
\item There are scalars $\xi^*,\zeta^*\in\mathbb{K}$ such that $\theta_i^{*\prime}=\xi^*\theta_{\rho+i}^*+\zeta^*$ $(0\leqslant i\leqslant d')$.
\item $\dfrac{\phi_i'}{\varphi_i'}=\dfrac{\phi_{\rho+i}}{\varphi_{\rho+i}}$ \ $(1\leqslant i\leqslant d')$.
\end{enumerate}
Moreover, if \textup{(i)}, \textup{(ii)} hold above then a bilinear form $(\cdot|\cdot):V\times V'\rightarrow\mathbb{K}$ which is $\rho$-balanced with respect to $\Phi$, $\Phi'$ is unique up to scalar multiplication.
\end{thm}

\begin{proof}
This is just a restatement of \eqref{A*', S*' in terms of A*sigma, S*sigma} and \eqref{bilinear form is balanced} in terms of the parameter arrays of $\Phi$, $\Phi'$.
The uniqueness follows from \eqref{values of products}.
\end{proof}

\begin{rem}
The endpoint $\rho$ is not necessarily uniquely determined by the parameter arrays of $\Phi$, $\Phi'$.
Indeed, with the notation of \eqref{list of parameter arrays} suppose that
\begin{equation*}
	p(\Phi)=p(\mathrm{IIC};r,s,s^*,\theta_0,\theta_0^*,d), \quad p(\Phi')=p(\mathrm{IIC};r,s,s^*,\theta_0',\theta_0^{*\prime},d').
\end{equation*}
Then conditions (i), (ii) in \eqref{characterization of bilinear form in terms of parameter arrays} are satisfied for \emph{all} $0\leqslant\rho\leqslant d-d'$.
\end{rem}

\section{Characterization of a balanced bilinear form in parametric form}\label{sec: bilinear form in parametric form}

In this section, we shall classify all the descendents of $\Phi$.
With reference to \eqref{Phi, Phi'}, we shall assume that $\Phi'$ is a $\rho$-descendent of $\Phi$ $(0\leqslant\rho\leqslant d-d')$, unless otherwise stated.
Let
\begin{equation*}
	\vartheta_i=\sum_{\ell=0}^{i-1}\frac{\theta_{\ell}-\theta_{d-\ell}}{\theta_0-\theta_d} \quad (1\leqslant i\leqslant d).
\end{equation*}
Clearly, $\vartheta_1=\vartheta_d=1$.
Moreover (\cite[Lemma 6.5]{Terwilliger2001LAA}):
\begin{equation}\label{how phi, varphi are related}
	\varphi_i-\phi_i=(\theta_i^*-\theta_{i-1}^*)(\theta_0-\theta_d)\vartheta_i \quad (1\leqslant i\leqslant d).
\end{equation}

\begin{note}\label{equations involving parameter arrays of Phi, Phi'}
With the notation of \eqref{characterization of bilinear form in terms of parameter arrays}\textup{(i)}, the following \textup{(i)}--\textup{(iii)} hold:
\begin{enumerate}
\item $(\theta_d-\theta_0)\vartheta_{\rho+i}\varphi_i'=\xi^*(\theta_{d'}'-\theta_0')\vartheta_i'\varphi_{\rho+i}$ \ $(1\leqslant i\leqslant d')$.
\item $(\theta_d-\theta_0)\vartheta_{\rho+i}\phi_i'=\xi^*(\theta_{d'}'-\theta_0')\vartheta_i'\phi_{\rho+i}$ \ $(1\leqslant i\leqslant d')$.
\item $(\theta_d-\theta_0)\vartheta_{\rho+1}\vartheta_{\rho+i}(\theta_{\rho+i}^*-\theta_{\rho}^*)(\theta_{d'-i+1}'-\theta_0')=(\theta_{d'}'-\theta_0')\vartheta_i'(\vartheta_{\rho+1}\phi_{\rho+i}-\vartheta_{\rho+i}\varphi_{\rho+1})$ \ $(1\leqslant i\leqslant d')$.
\end{enumerate}
\end{note}

\begin{proof}
(i), (ii):
Evaluate \eqref{characterization of bilinear form in terms of parameter arrays}(ii) in two ways using \eqref{how phi, varphi are related} and \eqref{characterization of bilinear form in terms of parameter arrays}(i).

(iii):
From \eqref{PA4} and \eqref{characterization of bilinear form in terms of parameter arrays}(i) it follows that
\begin{equation*}
	\phi_i'=\varphi_1'\vartheta_i'+\xi^*(\theta_{\rho+i}^*-\theta_{\rho}^*)(\theta_{d'-i+1}'-\theta_0').
\end{equation*}
On multiplying both sides above by $(\theta_d-\theta_0)\vartheta_{\rho+1}\vartheta_{\rho+i}$ and simplifying the result using (i) and (ii), we obtain (iii).
\end{proof}

We shall need the explicit values of the $\vartheta_i$.
With the notation of \eqref{list of parameter arrays} we have
\begin{equation}\label{vartheta}
	\vartheta_i=\begin{cases} \dfrac{(q^i-1)(q^{d-i+1}-1)}{(q-1)(q^d-1)} & \text{\emph{for Cases I, IA}}, \\ i(d-i+1)/d  & \text{\emph{for Cases II, IIA, IIB, IIC},} \\ i/d & \text{\emph{for Case III},}\ d \ \text{\emph{even},}\ i \ \text{\emph{even},} \\ (d-i+1)/d & \text{\emph{for Case III},}\ d \ \text{\emph{even},}\ i \ \text{\emph{odd},} \\ 0 & \!\!\!\begin{array}[t]{l} \text{\emph{for Case III},}\ d \ \text{\emph{odd},}\ i \ \text{\emph{even};} \\[-1mm] \text{\emph{or Case IV},}\ i \ \text{\emph{even},} \end{array} \\ 1 & \!\!\!\begin{array}[t]{l} \text{\emph{for Case III},}\ d \ \text{\emph{odd},}\ i \ \text{\emph{odd};} \\[-1mm] \text{\emph{or Case IV},}\ i \ \text{\emph{odd}} \end{array} \end{cases} \quad (1\leqslant i\leqslant d).
\end{equation}
(See e.g. \cite[Lemma 10.2]{Terwilliger2001LAA}.)
It should be remarked that from \eqref{PA1} and \eqref{PA2} we obtain restrictions on the scalar $q$ for Cases I, IA, and on the characteristic of $\mathbb{K}$ for Cases II, IIA, IIB, IIC and III.
We can then see that
\begin{note}\label{when vartheta=0}
We have $\vartheta_i=0$ precisely for Case III, $d$ odd, $i$ even; or Case IV, $i$ even.
\end{note}

Henceforth let $\beta+1$ denote the common value of \eqref{PA5}.
By convention, if $d<3$ then $\beta$ can be taken to be \emph{any} scalar in $\mathbb{K}$.
We may remark that
\begin{equation}\label{beta}
	\beta=\begin{cases} q+q^{-1} & \text{\emph{for Cases I, IA},} \\ 2 & \text{\emph{for Cases II, IIA, IIB, IIC},} \\ -2 & \text{\emph{for Case III},} \\ 0 & \text{\emph{for Case IV}.} \end{cases}
\end{equation}
It follows from \eqref{characterization of bilinear form in terms of parameter arrays}(i) that
\begin{equation}\label{beta'=beta}
	\beta'=\beta.
\end{equation}
For Case III, we have
\begin{note}\label{restrictions for Case III}
If $p(\Phi)$ satisfies Case III, then the following \textup{(i)}, \textup{(ii)} hold:
\begin{enumerate}
\item If $d$ is even, then either $d'=1$ or $d'$ is even.
\item If $d$ is odd, then $d'$ is odd and $\rho$ is even.
\end{enumerate}
\end{note}

\begin{proof}
From \eqref{PA1} and \eqref{equations involving parameter arrays of Phi, Phi'}(i) it follows that  $\vartheta_i'=0$ precisely when $\vartheta_{\rho+i}=0$, for $1\leqslant i\leqslant d'$.
Hence the result follows from \eqref{when vartheta=0}--\eqref{beta'=beta}.
\end{proof}

\noindent
Likewise,
\begin{note}\label{restrictions for Case IV}
If $p(\Phi)$ satisfies Case IV, then $(d',\rho)\in\{(1,0),(1,2),(3,0)\}$.
\end{note}

\begin{thm}\label{characterization of bilinear form in parametric form}
With reference to \eqref{Phi, Phi'}, let the parameter array of $\Phi$ be given as in \eqref{list of parameter arrays}.
Then $\Phi'$ is a $\rho$-descendent of $\Phi$ precisely when the parameter array of $\Phi'$ takes the following form:

\medskip\noindent
Case I:
\begin{equation*}
	p(\Phi')=p(\mathrm{I};q,h',h^{*\prime},r_1q^{\rho},r_2q^{\rho},sq^{d-d'},s^*q^{2\rho},\theta_0',\theta_0^{*\prime},d').
\end{equation*}
Case IA:
\begin{equation*}
	p(\Phi')=p(\mathrm{IA};q,h^{*\prime},r',s',\theta_0',\theta_0^{*\prime},d') \quad \text{where} \quad s'/r'=q^{d-d'-\rho}s/r.
\end{equation*}
Case II:
\begin{equation*}
	p(\Phi')=p(\mathrm{II};h',h^{*\prime},r_1+\rho,r_2+\rho,s+d-d',s^*+2\rho,\theta_0',\theta_0^{*\prime},d').
\end{equation*}
Case IIA:
\begin{equation*}
	p(\Phi')=p(\mathrm{IIA};h',r+\rho,s+d-d',s^{*\prime},\theta_0',\theta_0^{*\prime},d').
\end{equation*}
Case IIB:
\begin{equation*}
	p(\Phi')=p(\mathrm{IIB};h^{*\prime},r+\rho,s',s^*+2\rho,\theta_0',\theta_0^{*\prime},d').
\end{equation*}
Case IIC:
\begin{equation*}
	p(\Phi')=p(\mathrm{IIC};r',s',s^{*\prime},\theta_0',\theta_0^{*\prime},d') \quad \text{where} \quad s's^{*\prime}/r'=ss^*/r.
\end{equation*}
Case III, $d$ even, $d'$ even, $\rho$ even; or Case III, $d$ odd, $d'$ odd, $\rho$ even:
\begin{equation*}
	p(\Phi')=p(\mathrm{III};h',h^{*\prime},r_1+\rho,r_2+\rho,s-d+d',s^*-2\rho,\theta_0',\theta_0^{*\prime},d').
\end{equation*}
Case III, $d$ even, $d'$ even, $\rho$ odd:
\begin{equation*}
	p(\Phi')=p(\mathrm{III};h',h^{*\prime},r_2+\rho,r_1+\rho,s-d+d',s^*-2\rho,\theta_0',\theta_0^{*\prime},d').
\end{equation*}
Case III, $d$ even, $d'=1$; or Case IV, $(d',\rho)\in\{(1,0),(1,2)\}$:
\begin{equation*}
	p(\Phi')=p(\mathrm{IIC};r',s',s^{*\prime},\theta_0',\theta_0^{*\prime},1) \quad \text{where} \quad s's^{*\prime}/r'=1-\phi_{\rho+1}/\varphi_{\rho+1}.
\end{equation*}
Case IV, $(d',\rho)=(3,0)$:
\begin{equation*}
	p(\Phi')=p(\mathrm{IV};h',h^{*\prime},r,s,s^*,\theta_0',\theta_0^{*\prime}).
\end{equation*}
Case III, $d$ even, $d'$ odd $\geqslant 3$; or Case III, $d$ odd, either $d'$ even or $\rho$ odd; or Case IV, $(d',\rho)\in\{(1,1),(2,0),(2,1)\}$: Does not occur.
\end{thm}

\begin{proof}
Suppose first that $\Phi'$ is a $\rho$-descendent of $\Phi$.
The last line follows from \eqref{restrictions for Case III} and \eqref{restrictions for Case IV}.
For Cases I, IA, II, IIA, IIB, IIC; or for Case III with $d$ even and $d'$ even, it is a straightforward matter to show that $p(\Phi')$ is given as in \eqref{characterization of bilinear form in parametric form} by evaluating \eqref{characterization of bilinear form in terms of parameter arrays}(i) and \eqref{equations involving parameter arrays of Phi, Phi'}(i)--(iii) using \eqref{vartheta}--\eqref{beta'=beta} and \eqref{list of parameter arrays}.
(For example, $h'=hq^{d'-d}(q^d-1)(\theta_{d'}'-\theta_0')/(q^{d'}-1)(\theta_d-\theta_0)$, $h^{*\prime}=\xi^*h^*q^{-\rho}$ for Case I.)
For Case III with $d$ odd, $d'$ odd, $\rho$ even, likewise we have
\begin{align*}
	\theta_i^{*\prime}&=\theta_0^{*\prime}+h^{*\prime}(s^*-2\rho-1+(1-s^*+2\rho+2i)(-1)^i) \hspace{-5mm} & (0\leqslant i\leqslant d'), \\
	\theta_i'&=\theta_0'+2h'(s-d+d'-1-i) & (0\leqslant i\leqslant d',\ i\ \text{odd}), \\
	\varphi_i'&=-4h'h^{*\prime}(i+r_1+\rho)(i+r_2+\rho) & (0\leqslant i\leqslant d',\ i\ \text{odd}), \\
	\phi_i'&=-4h'h^{*\prime}(i-s^*+\rho-r_1)(i-s^*+\rho-r_2) & (0\leqslant i\leqslant d',\ i\ \text{odd}),
\end{align*}
where $h'=h(\theta_{d'}'-\theta_0')(\theta_d-\theta_0)^{-1}$, $h^{*\prime}=\xi^*h^*$.
Since $\beta'=-2$ by \eqref{beta'=beta} we see that
\begin{equation*}
	\theta_i'=\theta_0'+2h'i \quad (0\leqslant i\leqslant d',\ i\ \text{even})
\end{equation*}
by induction on \emph{even} $i$.
Hence from \eqref{PA3} and \eqref{PA4} it follows that $p(\Phi')$ is given as in \eqref{characterization of bilinear form in parametric form}.
The same argument applies to Case IV with $(d',\rho)=(3,0)$.
For Case III with $d$ even and $d'=1$; or for Case IV with $(d',\rho)\in\{(1,0),(1,2)\}$, we define $s'=\theta_1'-\theta_0'$, $s^{*\prime}=\theta_1^{*\prime}-\theta_0^{*\prime}$, $r'=-\varphi_1'$.
Then by \eqref{PA4} we have $\phi_1'=-(r'-s's^{*\prime})$, so that $p(\Phi')=p(\mathrm{IIC};r',s',s^{*\prime},\theta_0',\theta_0^{*\prime},1)$.
Furthermore, from \eqref{characterization of bilinear form in terms of parameter arrays}(ii) it is clear that $s's^{*\prime}/r'=1-\varphi_1'/\phi_1'=1-\varphi_{\rho+1}/\phi_{\rho+1}$.

Conversely, suppose that $p(\Phi')$ is of the form in \eqref{characterization of bilinear form in parametric form}.
Then it is easy to check \eqref{characterization of bilinear form in terms of parameter arrays}(i), (ii) and therefore $\Phi'$ is a $\rho$-descendent of $\Phi$.
This completes the proof.
\end{proof}

\section{Characterization of $\Phi'$ in terms of a balanced bilinear form}\label{sec: Phi' in terms of bilinear form}

The goal of this section is to  characterize the Leonard system $\Phi'$ in terms of the balanced bilinear form $(\cdot|\cdot)$.
We shall refer to the following set-up:

\begin{conv}\label{assumption on d', rho}
Let $\Phi$ be the Leonard system \eqref{Leonard system} and let the parameter array of $\Phi$ be given as in \eqref{list of parameter arrays}.
Let $d'$ be a positive integer such that $d'\leqslant d$, $\mathscr{A}'$ a $\mathbb{K}$-algebra isomorphic to $\mathrm{Mat}_{d'+1}(\mathbb{K})$, and $V'$ an irreducible left $\mathscr{A}'$-module.
Let $\rho$ be an integer such that $0\leqslant\rho\leqslant d-d'$.
We shall assume \textup{(i)}, \textup{(ii)} below:
\begin{enumerate}
\item For Case III, if $d$ is even then either $d'=1$ or $d'$ is even; if $d$ is odd then $d'$ is odd and $\rho$ is even.
\item For Case IV, $(d',\rho)\in\{(1,0),(1,2),(3,0)\}$.
\end{enumerate}
\end{conv}

\begin{note}\label{existence of Phi'}
There is a $\rho$-descendent of $\Phi$ with diameter $d'$.
\end{note}

\begin{proof}
We shall invoke \eqref{characterization of bilinear form in parametric form}.
Suppose first that we are in Case III with $d$ even and $d'=1$; or in Case IV with $(d',\rho)\in\{(1,0),(1,2)\}$.
From \eqref{how phi, varphi are related} and \eqref{when vartheta=0} it follows that $\varphi_{\rho+1}\ne\phi_{\rho+1}$ so that $s's^{*\prime}/r'=1-\phi_{\rho+1}/\varphi_{\rho+1}\not\in\{0,1\}$.
Hence $p(\mathrm{IIC};r',s',s^{*\prime},\theta_0',\theta_0^{*\prime},1)$ satisfies \eqref{PA1}--\eqref{PA5}.
For the other cases, the feasibility of the parameter array in \eqref{characterization of bilinear form in parametric form} can be directly checked from \cite[Examples 5.3--5.15]{Terwilliger2005DCC}.
\end{proof}

A \emph{decomposition} of $V'$ shall mean a sequence $\{U_i'\}_{i=0}^{d'}$ of one-dimensional subspaces of $V'$ such that $V'=U_0'+U_1'+\dots+U_{d'}'$ (direct sum).
The following theorem will prove useful in the study of $Q$-polynomial distance-regular graphs (cf. \cite{Tanaka2008pre}):

\begin{thm}\label{characterization of Phi' in terms of bilinear form}
With reference to \eqref{assumption on d', rho}, let $\{U_i'\}_{i=0}^{d'}$, $\{U_i^{*\prime}\}_{i=0}^{d'}$ be decompositions of $V'$.
Assume that there is a bilinear form $(\cdot|\cdot):V\times V'\rightarrow\mathbb{K}$ satisfying \textup{(i)}--\textup{(iii)} below:
\begin{enumerate}
\item $(E_i^*V|U_j^{*\prime})=0$ if $i-\rho\ne j$ \ $(0\leqslant i\leqslant d, \ 0\leqslant j\leqslant d')$.
\item $(E_iV|U_j')=0$ if $i<j$ or $i>j+d-d'$ \ $(0\leqslant i\leqslant d, \ 0\leqslant j\leqslant d')$.
\item $(\cdot|\cdot)$ has full-rank.
\end{enumerate}
Then there is a $\rho$-descendent $\Phi'=\bigl(A';A^{*\prime};\{E_i'\}_{i=0}^{d'};\{E_i^{*\prime}\}_{i=0}^{d'}\bigr)$ of $\Phi$ in $\mathscr{A}'$ such that $E_i'V'=U_i'$, $E_i^{*\prime}V'=U_i^{*\prime}$, for $0\leqslant i\leqslant d'$.
In particular, $(\cdot|\cdot)$ is $\rho$-balanced with respect to $\Phi$, $\Phi'$.
Moreover, $\Phi'$ is unique up to affine transformation.
\end{thm}

\begin{proof}
The uniqueness is easily verified, e.g. from \eqref{ith component of 0*}.
To show the existence, let $\Phi''=\bigl(A'';A^{*\prime\prime};\{E_i''\}_{i=0}^{d'};\{E_i^{*\prime\prime}\}_{i=0}^{d'}\bigr)$ be a $\rho$-descendent of $\Phi$ in $\mathscr{A}'$ and let $(\cdot|\cdot)_0: V\times V'\rightarrow\mathbb{K}$ be a bilinear form which is $\rho$-balanced with respect to $\Phi$, $\Phi''$.
Let $\mathrm{proj}':V'\rightarrow V$ (resp. $\mathrm{proj}'':V'\rightarrow V$) be a unique linear map satisfying $(v|v')=\langle v,\mathrm{proj}'v'\rangle$ (resp. $(v|v')_0=\langle v,\mathrm{proj}''v'\rangle$) for all $v\in V$, $v'\in V'$.
Since $\mathrm{proj}'$, $\mathrm{proj}''$ are injective (cf. \eqref{values of products}) we have
\begin{equation*}
	\mathrm{proj}' V'=\mathrm{proj}'' V'=E_{\rho}^*V+\dots+E_{\rho+d'}^*V,
\end{equation*}
from which it follows that there is a unique invertible element $C'$ in $\mathscr{A}'$ such that $\mathrm{proj}'C'v'=\mathrm{proj}''v'$ for all $v'\in V'$.
Let $\gamma:\mathscr{A}'\rightarrow\mathscr{A}'$ be the automorphism of $\mathscr{A}'$ defined by $X^{\prime\gamma}=C'X'C^{\prime-1}$ $(X'\in\mathscr{A}')$ and define $\Phi'=\Phi^{\prime\prime\gamma}$.
It remains to show that $E_i'V'=U_i'$, $E_i^{*\prime}V'=U_i^{*\prime}$ for $0\leqslant i\leqslant d'$.
On one hand, from (i) it is clear that $\mathrm{proj}'U_i^{*\prime}=E_{\rho+i}^*V$.
On the other hand, we have
\begin{equation*}
	\mathrm{proj}'E_i^{*\prime}V'=\mathrm{proj}'C'E_i^{*\prime\prime}C^{\prime-1}V'=\mathrm{proj}''E_i^{*\prime\prime}V'=E_{\rho+i}^*V.
\end{equation*}
Hence $E_i^{*\prime}V'=U_i^{*\prime}$.
Likewise we have $\mathrm{proj}'E_i'V'=\mathrm{proj}''E_i''V'$.
Furthermore,
\begin{equation*}
	\mathrm{proj}''E_i''V',\ \mathrm{proj}'U_i'\subseteq(E_{\rho}^*V+\dots+E_{\rho+d'}^*V)\cap(E_iV+\dots+E_{i+d-d'}V).
\end{equation*}
Consequently, it is enough to prove the following:
\begin{equation*}
	\dim(E_{\rho}^*V+\dots+E_{\rho+d'}^*V)\cap(E_iV+\dots+E_{i+d-d'}V)=1 \quad (0\leqslant i\leqslant d'),
\end{equation*}
for it would imply that $\mathrm{proj}'E_i'V'=\mathrm{proj}''E_i''V'=\mathrm{proj}'U_i'$ and hence that $E_i'V'=U_i'$.
For this purpose, as in the proof of  \eqref{bilinear form is balanced}(i) we observe that
\begin{align*}
	\langle E_iV,\mathrm{proj}''E_i''V'\rangle&=(E_iV|E_i''V')' \\
	&=(E_0V+\dots+E_iV|E_i''V'+\dots+E_{d'}''V')' \\
	&=(E_0V|E_0''V'+\dots+E_{d'}''V')' \\
	&\ne 0 && \text{by \eqref{projE0V=E0'V'}}.
\end{align*}
Since $(\cdot|\cdot)_0$ is also balanced with respect to $\Phi^{\Downarrow}, \Phi^{\prime\prime\Downarrow}$ (cf. \eqref{balanced with respect to Leonard pairs}) we have
\begin{equation*}
	\langle E_{i+d-d'}V,\mathrm{proj}''E_i''V'\rangle=\langle E_{d'-i}(\Phi^{\Downarrow})V,\mathrm{proj}''E_{d'-i}(\Phi^{\prime\prime\Downarrow})V'\rangle\ne 0.
\end{equation*}
Now let $v_i$ be a nonzero vector in $\mathrm{proj}''E_i''V'$, for $0\leqslant i\leqslant d'$.
Then $\{v_i\}_{i=0}^{d'}$ is a basis of $E_{\rho}^*V+\dots+E_{\rho+d'}^*V$.
Let $v\in(E_{\rho}^*V+\dots+E_{\rho+d'}^*V)\cap(E_iV+\dots+E_{i+d-d'}V)$.
Then $v$ is a linear combination of  $v_0,v_1,\dots,v_{d'}$, but no $v_j$ with $0\leqslant j<i$ can be involved since $v_j$ is not orthogonal to $E_jV$.
Likewise no $v_j$ with $i<j\leqslant d'$ can be involved since $v_j$ is not orthogonal to $E_{j+d-d'}V$.
It follows that $v$ is a scalar multiple of $v_i$, and therefore the proof is complete.
\end{proof}

\section{Remarks: balanced bilinear forms and polynomials from the Askey scheme}\label{sec: polynomials}

In this section, we return to the situation of \eqref{Phi, Phi'}.
Let $x$ be an indeterminate.
For $0\leqslant i\leqslant d$ we define a polynomial $u_i\in\mathbb{K}[x]$ by
\begin{equation}
	u_i=\sum_{\ell=0}^i\frac{(\theta_i^*-\theta_0^*)\dots(\theta_i^*-\theta_{\ell-1}^*)(x-\theta_0)\dots(x-\theta_{\ell-1})}{\varphi_1\varphi_2\dots\varphi_{\ell}}.
\end{equation}
The $u_i$ belong to the terminating branch of the Askey scheme \cite{KS1998R}, consisting of the $q$-Racah, $q$-Hahn, dual $q$-Hahn, $q$-Krawtchouk, dual $q$-Krawtchouk, quantum $q$-Krawtchouk, affine $q$-Krawtchouk, Racah, Hahn, dual Hahn, Krawtchouk, Bannai/Ito and orphan polynomials.
See \cite[Examples 5.3--5.15]{Terwilliger2005DCC}.

Let $u,u',v,v'$ be nonzero vectors in $E_0V,E_0'V',E_0^*V,E_0^{*\prime}V'$, respectively.
It is known \cite[Theorem 15.8]{Terwilliger2004LAA} that
\begin{equation}\label{Eiv in Ej*u}
	E_iv=k_i^*\frac{\langle u,v\rangle}{||u||^2}\sum_{j=0}^du_i^*(\theta_j^*)E_j^*u, \quad E_i'v'=k_i^{*\prime}\frac{\langle u',v'\rangle'}{||u'||^{\prime 2}}\sum_{j=0}^{d'}u_i^{*\prime}(\theta_j^{*\prime})E_j^{*\prime}u'.
\end{equation}
Suppose now that $\Phi'$ is a $\rho$-descendent of $\Phi$ and let $(\cdot|\cdot): V\times V'\rightarrow\mathbb{K}$ be a corresponding balanced bilinear form.
Then from \eqref{values of products} it follows that
\begin{align*}
	(E_iv|E_j'v')&=k_i^*k_j^{*\prime}\frac{\langle u,v\rangle}{||u||^2}\frac{\langle u',v'\rangle'}{||u'||^{\prime 2}}\sum_{\ell=0}^{d'}u_i^*(\theta_{\rho+\ell}^*)u_j^{*\prime}(\theta_{\ell}^{*\prime})(E_{\rho+\ell}^*u|E_{\ell}^{*\prime}u') \\
	&=\frac{\epsilon k_i^*k_j^{*\prime}}{\nu'}\frac{\langle u,v\rangle\langle u',v'\rangle'}{||u||^2}\sum_{\ell=0}^{d'}u_i^*(\theta_{\rho+\ell}^*)u_j^{*\prime}(\theta_{\ell}^{*\prime})k_{\ell}'.
\end{align*}
Consequently,
\begin{equation}\label{orthogonality}
	\sum_{\ell=0}^{d'}u_i^*(\theta_{\rho+\ell}^*)u_j^{*\prime}(\theta_{\ell}^{*\prime})k_{\ell}'=0 \quad\text{\emph{if}}\ i<j\ \text{\emph{or}}\ i>j+d-d'.
\end{equation}
Conversely, it is easy to show that the orthogonality \eqref{orthogonality} in turn characterizes the existence of a balanced bilinear form.
We may remark that \eqref{orthogonality} was previously observed by Hosoya and Suzuki \cite[Proposition 1.3]{HS2007EJC} for Leonard systems arising from certain pairs of $Q$-polynomial distance-regular graphs.

We illustrate \eqref{orthogonality} with an example.
With the notation of \eqref{list of parameter arrays}, assume that
\begin{equation*}
	p(\Phi)=p(\mathrm{IIC};r,s,s^*,\theta_0,\theta_0^*,d), \quad p(\Phi')=p(\mathrm{IIC};r,s,s^*,\theta_0',\theta_0^{*\prime},d').
\end{equation*}
It follows that the $u_i^*$, $u_i^{*\prime}$ are Krawtchouk polynomials:
\begin{equation*}
	u_i^*(\theta_j^*)=\hypergeometricseries{2}{1}{-i,-j}{-d}{\frac{1}{p}}, \quad u_i^{*\prime}(\theta_j^{*\prime})=\hypergeometricseries{2}{1}{-i,-j}{-d'}{\frac{1}{p}}
\end{equation*}
where $p=r/ss^*$ \cite[Example 5.13]{Terwilliger2005DCC}.
By \eqref{characterization of bilinear form in parametric form}, $\Phi'$ is a $\rho$-descendent of $\Phi$ for all $0\leqslant\rho\leqslant d-d'$.
Using \cite[Theorem 17.11]{Terwilliger2004LAA} we obtain
\begin{equation*}
	\sum_{\ell=0}^{d'}\binom{d'}{\ell}\!\!\left(\frac{p}{1-p}\right)^{\ell}\hypergeometricseries{2}{1}{-i,-\rho-\ell}{-d}{\frac{1}{p}}\hypergeometricseries{2}{1}{-j,-\ell}{-d'}{\frac{1}{p}}=0
\end{equation*}
whenever $i<j$ or $i>j+d-d'$.
This orthogonality for Krawtchouk polynomials can also be easily derived from \cite[Proposition 2.1]{CS1990JAT}.
However, it should be remarked that \eqref{characterization of bilinear form in parametric form} gives a complete classification of those polynomials satisfying \eqref{orthogonality} within the terminating branch of the Askey scheme.

\appendix

\section{The list of parameter arrays}

In this appendix we display the parameter arrays of the Leonard systems.
The data in \eqref{list of parameter arrays} is taken from \cite{Terwilliger2005DCC}, but for future use (cf. \cite{Tanaka2008pre}) the presentation is changed so as to be consistent with the notation in \cite{BI1984B,Terwilliger1992JAC,Terwilliger1993JACa,Terwilliger1993JACb}.
In \eqref{list of parameter arrays} the following implicit assumptions apply:
the scalars $\theta_i, \theta_i^*, \varphi_i,\phi_i$ are contained in $\mathbb{K}$, and the scalars $q,h,h^*,\dots$ are contained in the algebraic closure of $\mathbb{K}$.

\begin{thm}[{\cite[Theorem 5.16]{Terwilliger2005DCC}}]\label{list of parameter arrays}
Let $\Phi$ be the Leonard system from \eqref{Leonard system} and let $p(\Phi)=\left(\{\theta_i\}_{i=0}^d;\{\theta_i^*\}_{i=0}^d;\{\varphi_i\}_{i=1}^d;\{\phi_i\}_{i=1}^d\right)$ be the parameter array of $\Phi$.
Then at least one of the following cases I, IA, II, IIA, IIB, IIC, III, IV hold:
\begin{description}
\item[\textup{(I)}] $p(\Phi)=p(\mathrm{I};q,h,h^*,r_1,r_2,s,s^*,\theta_0,\theta_0^*,d)$ where $r_1r_2=ss^*q^{d+1}$,
\begin{align*}
	\theta_i&=\theta_0+h(1-q^i)(1-sq^{i+1})q^{-i}, \\
	\theta_i^*&=\theta_0^*+h^*(1-q^i)(1-s^*q^{i+1})q^{-i}
\end{align*}
for $0\leqslant i\leqslant d$, and
\begin{align*}
	\varphi_i&=hh^*q^{1-2i}(1-q^i)(1-q^{i-d-1})(1-r_1q^i)(1-r_2q^i), \\
	\phi_i&=\begin{cases} hh^*q^{1-2i}(1-q^i)(1-q^{i-d-1})(r_1-s^*q^i)(r_2-s^*q^i)/s^* & \text{if} \ s^*\ne 0, \\ hh^*q^{d+2-2i}(1-q^i)(1-q^{i-d-1})(s-r_1q^{i-d-1}-r_2q^{i-d-1}) & \text{if} \ s^*=0 \end{cases}
\end{align*}
for $1\leqslant i\leqslant d$.
\item[\textup{(IA)}] $p(\Phi)=p(\mathrm{IA};q,h^*,r,s,\theta_0,\theta_0^*,d)$ where
\begin{align*}
	\theta_i&=\theta_0-sq(1-q^i), \\
	\theta_i^*&=\theta_0^*+h^*(1-q^i)q^{-i}
\end{align*}
for $0\leqslant i\leqslant d$, and
\begin{align*}
	\varphi_i&=-rh^*q^{1-i}(1-q^i)(1-q^{i-d-1}), \\
	\phi_i&=h^*q^{d+2-2i}(1-q^i)(1-q^{i-d-1})(s-rq^{i-d-1})
\end{align*}
for $1\leqslant i\leqslant d$.
\item[\textup{(II)}] $p(\Phi)=p(\mathrm{II};h,h^*,r_1,r_2,s,s^*,\theta_0,\theta_0^*,d)$ where $r_1+r_2=s+s^*+d+1$,
\begin{align*}
	\theta_i&=\theta_0+hi(i+1+s), \\
	\theta_i^*&=\theta_0^*+h^*i(i+1+s^*)
\end{align*}
for $0\leqslant i\leqslant d$, and
\begin{align*}
	\varphi_i&=hh^*i(i-d-1)(i+r_1)(i+r_2), \\
	\phi_i&=hh^*i(i-d-1)(i+s^*-r_1)(i+s^*-r_2)
\end{align*}
for $1\leqslant i\leqslant d$.
\item[\textup{(IIA)}] $p(\Phi)=p(\mathrm{IIA};h,r,s,s^*,\theta_0,\theta_0^*,d)$ where
\begin{align*}
	\theta_i&=\theta_0+hi(i+1+s), \\
	\theta_i^*&=\theta_0^*+s^*i
\end{align*}
for $0\leqslant i\leqslant d$, and
\begin{align*}
	\varphi_i&=hs^*i(i-d-1)(i+r), \\
	\phi_i&=hs^*i(i-d-1)(i+r-s-d-1)
\end{align*}
for $1\leqslant i\leqslant d$.
\item[\textup{(IIB)}] $p(\Phi)=p(\mathrm{IIB};h^*,r,s,s^*,\theta_0,\theta_0^*,d)$ where
\begin{align*}
	\theta_i&=\theta_0+si, \\
	\theta_i^*&=\theta_0^*+h^*i(i+1+s^*)
\end{align*}
for $0\leqslant i\leqslant d$, and
\begin{align*}
	\varphi_i&=h^*si(i-d-1)(i+r), \\
	\phi_i&=-h^*si(i-d-1)(i+s^*-r)
\end{align*}
for $1\leqslant i\leqslant d$.
\item[\textup{(IIC)}] $p(\Phi)=p(\mathrm{IIC};r,s,s^*,\theta_0,\theta_0^*,d)$ where
\begin{align*}
	\theta_i&=\theta_0+si, \\
	\theta_i^*&=\theta_0^*+s^*i
\end{align*}
for $0\leqslant i\leqslant d$, and
\begin{align*}
	\varphi_i&=ri(i-d-1), \\
	\phi_i&=(r-ss^*)i(i-d-1)
\end{align*}
for $1\leqslant i\leqslant d$.
\item[\textup{(III)}] $p(\Phi)=p(\mathrm{III};h,h^*,r_1,r_2,s,s^*,\theta_0,\theta_0^*,d)$ where $r_1+r_2=-s-s^*+d+1$,
\begin{align*}
	\theta_i&=\theta_0+h(s-1+(1-s+2i)(-1)^i), \\
	\theta_i^*&=\theta_0^*+h^*(s^*-1+(1-s^*+2i)(-1)^i)
\end{align*}
for $0\leqslant i\leqslant d$, and
\begin{align*}
	\varphi_i&=\begin{cases} -4hh^*i(i+r_1) & \text{if $i$ even, $d$ even}, \\ -4hh^*(i-d-1)(i+r_2) & \text{if $i$ odd, $d$ even}, \\ -4hh^*i(i-d-1) & \text{if $i$ even, $d$ odd}, \\ -4hh^*(i+r_1)(i+r_2) & \text{if $i$ odd, $d$ odd}, \end{cases} \\
	\phi_i&=\begin{cases} 4hh^*i(i-s^*-r_1) & \text{if $i$ even, $d$ even}, \\ 4hh^*(i-d-1)(i-s^*-r_2) & \text{if $i$ odd, $d$ even}, \\ -4hh^*i(i-d-1) & \text{if $i$ even, $d$ odd}, \\ -4hh^*(i-s^*-r_1)(i-s^*-r_2) & \text{if $i$ odd, $d$ odd} \end{cases}
\end{align*}
for $1\leqslant i\leqslant d$.
\item[\textup{(IV)}] $p(\Phi)=p(\mathrm{IV};h,h^*,r,s,s^*,\theta_0,\theta_0^*)$ where $\mathrm{char}(\mathbb{K})=2$, $d=3$, and
\begin{align*}
	\theta_1&=\theta_0+h(s+1), & \theta_2&=\theta_0+h, & \theta_3&=\theta_0+hs, \\
	\theta_1^*&=\theta_0^*+h^*(s^*+1), & \theta_2^*&=\theta_0^*+h^*, & \theta_3^*&=\theta_0^*+h^*s^*, \\
	\varphi_1&=hh^*r, & \varphi_2&=hh^*, & \varphi_3&=hh^*(r+s+s^*), \\
	\phi_1&=hh^*(r+s(1+s^*)), & \phi_2&=hh^*, & \phi_3&=hh^*(r+s^*(1+s)).
\end{align*}
\end{description}
\end{thm}


\medskip
\begin{ack}
The author would like to thank Tatsuro Ito and Paul Terwilliger for a lot of fruitful discussions and comments concerning the paper.
Thanks are also due to the anonymous referee who provided helpful suggestions.
Part of this work was done while the author was visiting the Department of Computational Science at Kanazawa University in the Fall of 2007.
\end{ack}


\end{document}